\documentclass[twoside]{irmaems}
\usepackage{amssymb}
\usepackage{amsmath}
\usepackage{latexsym}
 \usepackage{makeidx}
\usepackage{graphicx}
\usepackage{epsfig}
\usepackage{amssymb}
\usepackage{amsmath}
\usepackage{amsfonts}
\usepackage{amscd}
\usepackage{amsfonts}
\usepackage{enumerate}
\input xy
\xyoption{all}

\renewcommand\theequation{\arabic{section}.\arabic{equation}}

\theoremstyle{definition} 

\theoremstyle{plain}      

\newtheorem{thm}{Theorem}[section]

\newtheorem{Prop}[thm]{Proposition}
\newtheorem{lemma}[thm]{Lemma}

\newtheorem{co}[thm]{Corollary}

\theoremstyle{definition}
\newtheorem{definition}[thm]{Definition}
\newcommand{\R}{\mathbb{R}}

\let\cal\mathcal

\newcommand{\br}{{\mathbb R}}

\newcommand{\ra}{\rightarrow}

\let \cal \mathcal

\makeindex

\begin{document}
\title{Convex real projective structures and Hilbert metrics}
 
\author{Inkang Kim\thanks{The first author gratefully acknowledges the partial support
of NRF grant  (2010-0024171).
} {} and Athanase Papadopoulos\thanks{The second author   is supported by the French ANR project FINSLER.
} }

\address{School of Mathematics,
KIAS\\ Heogiro 85, Dongdaemun-gu Seoul, 130-722, Republic of Korea\\
email:\,\tt{inkang@kias.re.kr}
\\
{}
\\
Institut de Recherche Math\'ematique Avanc\'ee
\\
Universit\'e de Strasbourg et CNRS
\\
7 rue Ren\'e Descartes
\\
67084 Strasbourg Cedex, France
\\
email:\,\tt{papadop@math.unistra.fr}}

\maketitle

\begin{abstract} 
We review some basic concepts related to  convex real projective structures from the differential geometry point of view. We start by recalling a Riemannian metric which originates in the study of affine spheres using the Blaschke connection (work of Calabi and of Cheng-Yau) mentioning its relation with the Hilbert metric. We then survey some of the deformation theory of convex real projective structures on surfaces. We describe in particular how the set of (Hilbert) lengths of simple closed curves is used in a parametrization of the deformation space in analogy with the classical Fenchel-Nielsen parameters of Teichm\"uller space (work of Goldman). We then mention parameters of this deformation space that arise in the work of Hitchin on the character variety of representations of the fundamental group of the surface in $\mathrm{SL}(3,\mathbb{R})$. In this character variety, the component of the character variety that corresponds to projective structures is identified with the vector space of pairs of holomorphic quadratic and cubic differentials over a fixed Riemann surface. Labourie and Loftin (independently) obtained parameter spaces that use the cubic differentials and affine spheres. We then display some similarities and differences between Hilbert geometry and  hyperbolic geometry using geodesic currents and topological entropy. Finally, we discuss geodesic flows associated to Hilbert metrics and compactifications of spaces of convex real projective structures on surfaces. This makes another analogy with works done on the Teichm\"uller space of the surface.
\end{abstract}



\tableofcontents

\section{Introduction}

In what follows,  $\mathbb{RP}^n$ is the  $n$-dimensional real projective space,  that is, the set of lines through the origin in $\mathbb{R}^{n+1}$ and $\mathbb{A}^n$ the $n$-dimensional affine space, considered as the complement of a hyperplane in $\mathbb{RP}^n$. A subset $\Omega$ of $\mathbb{A}^n$ is \emph{convex} if its intersection with each affine line is connected.

Let $\Omega$ be an open convex subset of $\mathbb{A}^n$.    Then $\Omega$ is equipped with a canonical metric, called the Hilbert metric, whose definition we recall in \S \ref{s:preliminaries} below. This metric is invariant by the group of projective transformations of  $\mathbb{RP}^n$ which preserve $\Omega$. Several interesting phenomena were discovered and several good questions arose recently concerning that metric.  It is non-Riemannian except in the case where $\Omega$ is an ellipsoid, but it shares several properties with the hyperbolic metric (that is, a Riemannian metric of constant negative curvature),\index{hyperbolic metric}\index{metric!hyperbolic} especially if $\Omega$ is strictly convex. There are also many differences. In the next sections, we will highlight some of these analogies and differences.

We start by recalling very classical questions concerning convex sets equipped with their Hilbert geometry. From the observation that the boundary of the unit disk is a smooth circle, it was natural to ask whether there exist convex domains in the projective plane with less regular boundary which admit  compact quotients by discrete subgroups of the automorphism group of the domain. This question is natural because the unit disk admits such quotients. This existence question was already addressed by Ehresmann in the 1930s \cite{Ehresmann},\index{Ehresmann (Charles)} and it was answered affirmatively thirty years later by Kac and Vinberg \cite{KV}. In modern language, a convex domain in the projective plane which admits a cocompact action is termed \emph{divisible}.\index{divisible} An ellipse in $\mathbb{RP}^2$ or an ellipsoid $\mathbb{RP}^3$ are well-known examples of convex domains. Their compact quotients are the 2- and 3-dimensional compact hyperbolic manifolds. The theory of hyperbolic manifolds (their construction, classification, their deformation and moduli spaces and their rigidity properties) was born in the works of Klein and Poincar\'e in the 1880s. This theory rapidly evolved into one of the most beautiful geometric theories that were developed during the twentieth century. Important questions concerning more general convex projective manifolds (quotients of divisible convex domains) were addressed by Benz\'ecri\index{Benz\'ecri (Jean-Paul)} in his thesis \cite{Ben} (1960). This thesis contains some foundational work on the subject. To find conditions under which a convex domain is divisible is a difficult matter, as is the general question of existence of lattices in semisimple Lie groups. In most of the cases, the boundary of such a domain is nowhere analytic. Benz\'ecri proved that if the boundary is $C^2$, then the original convex set is an ellipsoid, which makes it identified with the hyperbolic space, more precisely, with the Cayley-Klein-Beltrami\index{Cayley-Klein-Beltrami model} model of that space. It also follows from the work of Benz\'ecri and from later works that the boundary $\partial \Omega$ of a divisible convex domain $\Omega$ is either a conic (in which case the projective structure arises from a hyperbolic structure) or this boundary is nowhere $C^{1+\epsilon}$ for some $\epsilon >0$. Explicit examples of quotients of projective manifolds obtained by reflections along convex polyhedra in the sphere were obtained by Vinberg in the early 1970s. These constructions may be considered as projective analogues of Poincar\'e's examples of hyperbolic manifolds that use reflections along convex polyhedra in hyperbolic space.

A \emph{convex real projective manifold}\index{convex projective structure}\index{projective structure!convex} is the quotient of a convex set $\Omega\subset\mathbb{RP}^n$ by a discrete group of projective transformations, and the structure is termed \emph{strictly convex}\index{strictly convex projective structure}\index{projective structure!strictly convex} if $\Omega$ is strictly convex. The moduli space of convex projective structures generalizes the Teichm\"uller spaces of hyperbolic surfaces. There are many interesting recent developments on convex projective manifolds and their deformations. Hitchin \cite{Hi} discovered a component in the space of representations of fundamental groups of surfaces in $\mathrm{SL}(n,\mathbb{R})$, for every $n\geq 2$,  which contains Teichm\"uller space, and whose elements have several properties in common with Teichm\"uller space. He called these  components Teichm\"uller components,\index{Teichm\"uller component} and today such a component is called a Hitchin component.\index{Hitchin component} In the case $n=3$, such components consist of real projective structures on surfaces. Johnson and Millson \cite{JM} showed that there are non-trivial continuous deformations of higher-dimensional hyperbolic structures through strictly convex projective structures.  Conversely, one might expect that if a compact manifold admits a strictly convex projective structure, then it is a deformation of a hyperbolic structure. But Benoist \cite{Be1} constructed in dimension four an example of a manifold which admits a strictly convex real projective structure but no hyperbolic structure. Later, Kapovich \cite{Ka} generalized the method to show that such examples exist in any dimension $\geq 4$. 

In this paper, we study some relations between hyperbolic geometry and differential  projective geometry. Some of the natural questions that appear in this setting are:
\begin{itemize}

\item To what extent does Hilbert geometry generalize hyperbolic geometry?
 
\item  What are the relations and the common properties between spaces of deformations of convex projective structures and Teichm\"uller spaces?
 
\item Are there compactifications of deformation spaces of convex  projective manifolds that are analogous to compactifications of Teichm\"uller spaces?
 
\item  How does the hyperbolic behavior of geodesic flows of Hilbert manifolds generalize the hyperbolic behavior (and in particular the Anosov theory) of negatively curved Riemannian manifolds?  
 \end{itemize}
 
  The exposition is by no means complete and thorough. In some cases, we just record the results together with some references known to us, hoping to arouse the reader's interest in these questions.

We would like to thank Sarah Bray, Bill Goldman, Ludovic Marquis and Marc Troyanov for valuable comments and corrections on a preliminary version of this paper.

\section{Preliminaries}\label{s:preliminaries}
We recall some basic notions of projective geometry on manifolds.

The automorphism group of $\mathbb{RP}^n$, denoted by $\mathrm{PGL}(n,\mathbb{R})$ -- the group of projective transformations -- is the  quotient of the linear group $\mathrm{GL}(n+1,\mathbb{R})$ by the action of the nonzero scalar transformations. It is sometimes convenient to work on the sphere $\mathbb{S}^n$, which is a double-sheeted cover of $\mathbb{RP}^n$. The sphere is equipped with the induced projective structure whose automorphism group (also called the projective automorphism group) is the group  $\mathrm{SL}^\pm(n+1,\mathbb{R})$ of real $(n+1)\times (n+1)$ matrices of determinant $\pm1$. This is also the group of volume-preserving affine transformations of $\mathbb{R}^{n+1}$.

The \emph{projective lines}\index{projective line} (or, more simply, the \emph{lines}) in $\mathbb{RP}^n$  are the projections of the great circles of the sphere $\mathbb{S}^n$ by the canonical map $\mathbb{S}^n\to \mathbb{RP}^n$. This set of lines is preserved by the group of projective transformations. The projective lines play simultaneously the role of \emph{lines} of a geometry defined in the axiomatic sense, and the role of geodesics for the Hilbert metric, as we shall recall  below.
  
Let $\Omega$ be a subset of $\mathbb{RP}^n$. We say that $\Omega$ is \emph{convex}\index{convex set!projective space} if the intersection of $\Omega$ with any line of $\mathbb{RP}^n$ is connected. We say that $\Omega$ is a \emph{properly convex}\index{properly convex set}\index{convex set!properly} (or a \emph{proper convex})\index{proper convex} subset of $\mathbb{RP}^n$ if it is convex and contained in the complement of a hyperplane. 

An \emph{affine patch}\index{affine patch}\index{patch!affine} in $\mathbb{RP}^n$ is the complement of a hyperplane. An affine patch can be seen as the affine $n$-dimensional space $\mathbb{A}^n$. 

Let $\Omega$ be a properly convex subset of $\mathbb{RP}^n$.
The convexity of $\Omega$ as  a subset of  projective space (that is, the intersection of $\Omega$ with each projectve line is connected) is equivalent to the convexity of $\Omega$ as a subset of affine space (that is, the intersection of $\Omega$ with each affine line is connected).
 
Let $\Omega$ be a properly convex subset of $\mathbb{RP}^n$ contained in an affine patch $A$. We equip $A$ with a Euclidean norm 
$| \cdot |$. We shall also denote by $| \cdot |$ the associated Euclidean distance. The space $A$ plays now the role of a Euclidean space $\mathbb{R}^n$ in which $\Omega$ sits. We present a brief summary of some notions associated to the Hilbert geometry of $\Omega$ that we shall use in this paper.

For $x\neq y\in \Omega$, let $p,q$ be the
intersection points of the line $xy$ with $\partial\Omega$ such that $p,x,y,q$ are in this order. The
\emph{Hilbert distance}\index{Hilbert metric} between $x$ and $y$ is defined by
$$d_\Omega(x,y)=\frac{1}{2}\log \frac{|p-y||q-x|}{|p-x||q-y|}$$. The value of $d_\Omega(x,y)$ does not depend on the choice of the Euclidean metric $\vert . \vert$ on $\mathbb{A}^n$. For $x=y$, we set $d_\Omega(x,y)=0$. 

This metric coincides with the familiar hyperbolic metric is the case where $\partial\Omega$ is an ellipse (in dimension 2) or ellipsoid (in dimension $\geq 3$). This metric of the ellipse or ellipsoid is the so-called ``projective model", or Cayley-Klein-Beltrami model of hyperbolic geometry.\index{Cayley-Klein-Beltrami model}

 The Hilbert metric is Finsler, and it is not Riemannian unless $\partial\Omega$ is an ellipsoid. The Finsler norm is given, for $x\in\Omega$ and a
vector $v$ in the tangent space of $\Omega$ at $x$, by
\begin{equation}\label{norm} 
 ||v||_x = \frac{1}{2} \left(\frac{1}{|x-p^-|}+\frac{1}{|x-p^+|}\right)|v|
\end{equation} 
 where $p^\pm$
are the intersection points with $\partial\Omega$ of the oriented line in $\Omega$ defined by the vector $v$ based at $x$ and where $| \cdot |$ is our
chosen norm on the affine patch.  This norm is reversible,\index{reversible norm} that is, it satisfies $||v||_x=||-v||_x$.
 The \emph{Finsler metric}\index{Finsler metric} associated to this Finsler nom is the metric on $\Omega$ defined by taking the distance between two arbitrary points to be the infimum of the lengths of $C^1$ paths  joining them, where the length of a path is defined by integrating the norms of tangent vectors using Formula (\ref{norm}). It is an easy exercise to show that this Finsler metric on $\Omega$ is the Hilbert metric.\index{Hilbert metric!Finsler norm} 
  
The \emph{(projective) automorphism}\index{convex set!projective automorphism group} group of $\Omega$ is the group of projective transformations of $\mathbb{R}^{n+1}$ that preserve $\Omega$. In other words, we have
\[\mathrm{Aut}(\Omega)=\{f\in \mathrm{PGL}(n,\mathbb{R}) \ \vert\  f(\Omega)=\Omega)\}.\]
The Hilbert metric of $\Omega$ is invariant by the group $\mathrm{Aut}(\Omega)$. A recent result by Walsh described in Chapter 5 of this volume (\cite{W}) says that the isometry group $\mathrm{Aut}(\Omega,d_\Omega)$ of the Hilbert metric $d_\Omega$ coincides with the group $\mathrm{Aut}(\Omega)$ except in a few special cases where $\mathrm{Aut}(\Omega,d_\Omega)$ is an order-two extension of $\mathrm{Aut}(\Omega)$.

 We now associate a Borel measure on $\Omega$ using the Finsler structure. The construction is analogous to the one used in Riemannian geometry. We follow the presentation of Marquis in \cite{Marquis-These}.

 Let $\mathrm {Vol}$ be a Lebesgue measure on $A$ normalized by
$\mathrm{Vol}(\{v\in A:|v|<1\})=1$. We define a measure on $\Omega$ by setting, for each Borel subset $\cal A\subset
\Omega\subset A$, 
$$\mu_\Omega(\cal A)=\int_{\cal A} \frac{d\mathrm {Vol}(x)}{\mathrm {Vol}(B_x(1))},$$
where $B_x(1)=\{v\in T_x\Omega: ||v||_x<1\}$, the norm $||.||$ being the one given by (\ref{norm}). 

This measure turns
out to be the Hausdorff measure induced by the Hilbert metric
\cite{BBI}. In particular, it is independent of the choice of the Euclidean norm of $A$. It is called the \emph{Busemann volume},\index{Busemann volume} and also the \emph{Hilbert volume}.\index{Hilbert volume} It is invariant by the action of $\mathrm{Aut}(\Omega)$.

From the definition, we have the following, for any two convex domains
$\Omega_1\subset \Omega_2$ of $\mathbb{R}^n$:
\begin{enumerate}
\item  $||v||_x^{\Omega_2}\leq ||v||_x^{\Omega_1}$ for every tangent vector $v$ based at a point $x$ in $\Omega_1$;
\item $d_{\Omega_2}(x,y)\leq d_{\Omega_1}(x,y)$ for every $x$ and $y$ in $\Omega_1$;
\item  $B_x^{\Omega_1}(1)\subset B_x^{\Omega_2}(1)$ for every $x$ in $\Omega_1$;
\item  $\mu_{\Omega_2}(\cal A)\leq \mu_{\Omega_1}(\cal A)$ for any
Borel set $\cal A$ in $\Omega_1$.
\end{enumerate}

For more details on the notion of measure associated to a general Finsler metric, we refer the reader to \cite{AT} and  \cite{BBI}. 

We now consider projective structures on manifolds.

\begin{definition}\label{d:structure}
A  \emph{real projective structure}\index{projective structure} on an $n$-dimensional differentiable manifold  $M$ is a maximal atlas with values in the $n$-dimensional projective space $\mathbb{RP}^n$ whose transition
functions are restrictions of projective automorphisms of $\mathbb{RP}^n$. 
\end{definition}
Equipped with such a structure, the manifold $M$ becomes a \emph{real projective manifold}. 
We shall sometimes delete the adjective ``real" since all the projective structures we consider in this paper are real.\index{projective manifold} An \emph{isomorphism}\index{projective manifold!isomorphism} between two $n$-dimensional projective manifolds is a homeomorphism between the underlying manifolds which, in each projective chart, is locally the restriction of  an element of the projective transformation group of $\mathbb{RP}^n$. 

To each real projective structure is associated, by the general theory of $G$-structures of Ehresmann \cite{Ehresmann} \index{Ehresmann (Charles)} a developing map and a holonomy homomorphism. The developing map arises from the attempt to define a global chart for the structure. It is constructed by starting with a coordinate chart around a given point and trying to extend it by analytic continuation. There is an obstruction for doing so if the manifold is not simply connected, and in fact, the only obstruction is the fundamental group of $M$. (When we come back to the point we started with, by analytically continuing along a nontrivial path, we generally end up with a different germ of a chart into $\mathbb{RP}^n$ than the one we started with.) Thus, instead of obtaining a map from our manifold $M$ to $\mathbb{RP}^n$, we end up with a map
\[ \mathrm{dev}: \tilde{M}\to \mathbb{RP}^n,
\] 
where  $\tilde{M}$ is the universal covering of $M$. The  holonomy homomorphism is an injective homomorphism 
\[hol: \pi_1(M)\to\mathrm{PSL}(n+1,\mathbb{R})\]
satisfying 
\[
 \mathrm{dev}\circ \gamma = \mathrm{hol}(\gamma)\circ \mathrm{dev}.
\]

This developing map is completely determined by its restriction to an arbitrary open subset of $M$. When we change the initial coordinate chart in the above construction, the resulting map differs from the previous one by post-composition by an element of $\mathrm{PGL}(n,\mathbb{R})$. The result is that although the developing map and the holonomy homomorphism depend on some choices (namely, the choice of the initial chart), there are nice transformation formulae relating the various maps and homomorphisms obtained. In particular, the holonomy homomorphism is well defined up to a conjugation by an element of $\mathrm{PGL}(n,\mathbb{R})$. 

In the case where the developing map is a homeomorphism onto its image, we can write $M=\mathrm{dev}(\tilde{M})/\mathrm{hol}(\pi_1(M))$. We refer the reader to the paper by Ehresmann \cite{Ehresmann} for the general theory of developing maps and holonomy representations associated to geometric structures. Benz\'ecri \cite{Ben},\index{Benz\'ecri (Jean-Paul)} Kuiper \cite{Kuiper} and subsequently Koszul \cite{Koszul} considered thoroughly the case of real projective structures.
Thurston, starting in the 1970s,  included this theory as an important part of the general theory of geometrization of low-dimensional manifolds, see  \cite{Thurston}. Goldman \cite{Goldman}, motivated by ideas of Thurston, developed the theory of moduli spaces of projective structures on surfaces. 
In their paper \cite{ST}, Sullivan  and Thurston give an example of a projective structure on the torus whose developing map is not a covering of projective space. 

Talking about the sources of differential projective geometry, one has also mention the work of Chern on the Gauss-Bonnet formula and characteristic classes that motivated several later works \cite{Chern}. 

In what follows, we shall use the more restrictive notion of \emph{convex}\index{projective structure!convex} real projective structure. This is the case where the developing map sends homeomorphically the universal cover $\tilde{M}$ of $M$ onto a convex subset of some $\mathbb{R}^n$ sitting in $\mathbb{RP}^n$ as the complement of an affine hyperplane.  It turns out that such a structure is more manageable than a general real projective structure, and in particular, one can introduce the Hilbert metric into the playground. There is a lot of classical and more recent activity on convex projective structures and their deformation spaces.

 In what follows, we shall use the terminology \emph{convex} to mean properly convex. There is a characterization of such a structure, which we can also take as a definition:

\begin{definition}\label{d:cstructure} A \emph{convex real projective manifold}\index{projective manifold!convex}\index{convex real projective manifold} $M$ is an object of the form $\Omega/\Gamma$ where $\Omega$ is a convex domain in $\mathbb{RP}^n$ containing no projective line and $\Gamma$ a discrete group of the projective automorphism group $\mathrm{Aut}(\Omega)$ of $\Omega$.
We call $M$ \emph{strictly convex}\index{projective manifold!convex!strictly} if $\Omega$ is strictly convex i.e., if $\Omega$ is convex and $\partial \Omega$ does not contain any line segment. 
\end{definition}

The relation between Definitions \ref{d:structure} and \ref{d:cstructure} is made through the natural identifications between $\Omega$ and the universal cover of $M$ and between $\Gamma$ and the fundamental group of $M$.

In dimension two, we have more knowledge about convex projective structures, and in particular there is the following in Goldman \cite{Goldman} (Proposition 3.1):
\begin{Prop} Let $M$ be a projective structure on a surface. Then, the following three properties are equivalent:
\begin{enumerate}
\item $M$ is projectively equivalent to a quotient $\Omega/\Gamma$ where $\Omega$ is a convex open subset of $\mathbb{RP} ^2$ and $\Gamma$ a discrete group of projective transformations of $\mathbb{RP} ^2$ which acts freely and properly discontinuously;
\item the developing map $\mathrm{dev}:\tilde{M}\to \mathbb{RP} ^2$ is a diffeomorphism onto a convex subset of $\mathbb{RP} ^2$;
\item every path in $M$ is homotopic relative endpoints to a geodesic path (that is, a path which in coordinate charts is contained in a line of $\mathbb{RP} ^2$).
\end{enumerate}
\end{Prop}

We shall be interested in quotients of convex sets, and we make right now the following definition:

\begin{definition}
An open (properly) convex  set $\Omega$ is said to be \emph{divisible}\index{convex set!divisible}\index{divisible!convex set} if there exists a discrete subgroup $\Gamma\subset \mathrm{Aut}(\Omega)$ such that  $\Omega/\Gamma$ is compact.
\end{definition}

 There is a natural identification between two convex real projective manifolds $\Omega_1/\Gamma_1$  and $\Omega_2/\Gamma_2$, defined by the condition that there exists a projective transformation $g$ of $\mathbb{RP}^n$ such that $g(\Omega_1)=\Omega_2$ and $g\Gamma_1g^{-1}=\Gamma_2$. This equivalence relation is used in the definition  of the \emph{moduli space}\index{convex structures!moduli space} and the \emph{deformation space}\index{convex structures!deformation space}  of convex real projective structures. 
 
 We are mainly interested in the deformation space of convex projective structures on surfaces, which, like  the Teichm\"uller space of hyperbolic (or of conformal) structures, is a space of equivalence classes of 
 \emph{marked}\index{convex projective structure!marked}\index{marked convex projective structure} 
 convex projective structures. We recall the concept of marking\index{marking}, which originates in Teichm\"uller theory.
 
 We start with a fixed topological surface $S_0$. 
 \begin{definition}
 A \emph{marked convex projective structure} on $S_0$ is a pair $(f,S)$ where $S$ is a surface homeomorphic to $S_0$ equipped with a convex projective structure, and  $f:S_0\to S$ is a homeomorphism.
 \end{definition}
  A marked convex projective structure on $S_0$ induces a convex projective structure on the base surface $S_0$ itself by pull-back. Conversely, a convex projective structure on $S_0$ can be considered as a marked convex projective structure, by taking the marking to be the identity homeomorphism of $S_0$. \begin{definition}
  The \emph{deformation space} of convex projective structures on $S_0$ is the set of equivalence classes of pairs $(f,S)$, where $S$ is a convex projective surface homeomorphic to $S_0$ and where two pairs $(f,S)$ and $(f',S')$ are considered to be equivalent if there exists a projective homeomorphism $f'':S\to S'$ that is homotopic to $f'\circ  f^{-1}$.
  \end{definition}

Equivalently, we can define the deformation space of convex projective structures on a surface (or, more generally, on a manifold) as the space of homotopy classes of convex projective structures on that surface (or manifold). This space is equipped with a natural topology arising from the $C^1$ topology on developing maps. This topology is Hausdorff (see \cite{Goldman} p. 793). We shall say more about the deformation space of real projective structures on surfaces in \S \ref{para}.

The basic elements of the theory of convex real projective structures on closed surfaces and their deformations are due to Kuiper \cite{Kuiper}, Kac-Vinberg \cite{KV} and Benz\'ecri \cite{Ben}, and a complete theory has been developed (including the case of surfaces with boundary) by Goldman \cite{Goldman}. We shall elaborate on Goldman's parametrization of the deformation space in \S \ref{para}.

For any convex real projective manifold $M=\Omega/\Gamma$, the Hilbert metric on $\Omega$ descends to a metric on $M$ called the Hilbert metric of $M$.

For a strictly convex real projective structure $M=\Omega/\Gamma$, as in the hyperbolic case, there exists a unique Hilbert geodesic in each homotopy class of a loop, unless this loop represents a parabolic element of the fundamental group, in which case the curve is homotopic to a puncture (or cusp) of $S$.
\begin{lemma}\label{lemma:m} For any convex real projective surface $S$ of finite type and of negative Euler characteristic (the surface may have geodesic boundary components and cusps), the area of
$S$ with respect to the Hausdorff measure induced by the Hilbert
metric is uniformly bounded below independently of the topology of
$S$.
\end{lemma}

The proof uses a pair of pants decomposition of $S$. We recall that a (topological) \emph{pair of pants}\index{pair of pants} in  $S$ is an embedded surface which is homeomorphic to a sphere with three holes, where a \emph{hole} is either a puncture or a boundary component, such that the following two conditions hold:
\begin{itemize}
\item Each boundary component of $P$ is a simple closed curve in $S$ which is not homotopic to a point or to puncture of $S$.
\item There is no embedded annulus in $S$ whose two boundary components are the union of  boundary component of $P$ and a boundary component of $S$.
\end{itemize}
Note that the holes of $P$ might be boundary components of $S$, and they can also be punctures, and the latter are considered as boundary curves of length zero.

A \emph{topological pair of pants decomposition}\index{pair of pants!decomposition!topological} $\mathcal{P}$ of $S$ is a union of disjoint simple closed curves in that surface such that the closure of each connected component of the complement of $\mathcal{P}$ in $S$ is a topological pair of pants in $S$.

Any surface of negative Euler characteristic admits a topological pair of pants decomposition. It is easy to see, using an Euler characteristic argument, that for a closed surface $S_g$ of genus $g\geq 2$, there are $2g-2$ pairs of pants in a pair of pants decomposition. Given a surface equipped with a metric, ta A \emph{geodesic pair of pants decomposition}\index{pair of pants!decomposition!geodesic} is a topological pair of pants decomposition in which each curve which is not homotopic to a puncture is a closed geodesic. On any surface with finitely generated fundamental group equipped with a hyperbolic metric or with a Hilbert metric, any topological pair of pants decomposition is homotopic to a geodesic pair of pants decomposition. Furthermore, for hyperbolic metrics and for strictly convex structures, the closed geodesics in each free homotopy classes are unique, so every  topological pair of pants decomposition is homotopic to a unique geodesic pair of pants decomposition. 

We now sketch the proof of Lemma \ref{lemma:m}.

\begin{proof}Take a geodesic pair of pants $P$ in $S$. One can decompose $P$ into two ideal triangles. Therefore, it suffices to give a 
bound for the area of an ideal triangle. Let $T$ be a lift of
an ideal triangle in $\Omega$ with three ideal vertices
$p_1,p_2,p_3$. Choose three projective lines $P_1,P_2,P_3$
containing $p_1,p_2,p_3$ disjoint from $\Omega$. They form a
triangle $\triangle$ containing $\Omega$. Now by a comparison
argument, it suffices to lower bound the area of $T$ in $\triangle$. This
is an exercise in projective geometry; see \cite{CV}. 
\end{proof}

There are several papers where the reader can find concise introductions to the basics of the modern theory of projective geometry. We refer to the survey by Benoist \cite{Benoist-survey} and to the sections on preliminaries in the thesis of Marquis \cite{Marquis-These}. 

Let us note that in studying divisible convex sets of finite co-volume, we do not lose a lot if we restrict ourselves to strictly convex sets. The following theorem is due to Marquis \cite{Marquis}.
\begin{thm}\label{finitevolume}Let $\Omega$ be a proper open convex subset of $\mathbb{RP}^2$. If $\Omega$ is not a triangle and 
admits a finite-volume quotient surface, then $\Omega$ is strictly convex.
\end{thm}

\section{Relation to affine spheres}\label{affine}

An affine sphere is a smooth hypersurface in $\mathbb{R}^{n+1}$ characterized by the condition that its affine normal lines meet at a common point. The family of affine spheres is invariant under the group of affine transformations of $\mathbb{R}^{n+1}$. In fact, it is the simplest interesting such family. Thus, it is not surprising that affine spheres are useful in the study of affine, and also of projective structures. (Recall that affine geometry is projective geometry where a hyperplane at infinity in $\mathbb{RP}^n$ has been fixed.) Affine spheres also appear naturally in the solutions of certain PDEs, namely, the Monge-Amp\`ere equations, and also in the study of hyperbolic surfaces. In fact, it is known that on a compact hyperbolic surface equipped with a cubic differential, there is a unique associated equiaffine\index{equiaffine metric} metric called the \emph{Blaschke metric}.\index{metric!Blaschke}\index{metric!equiaffine} (An object is called \emph{equiaffine} if it is invariant by the group of volume-preserving affine transformations.) This, together with the Cheng-Yau classification of complete hyperbolic affine spheres, gives a new parametrization of the space of real projective structures on the surface and of the Hitchin component of the representations of the fundamental group of the surface into $\mathbb{SL}(3,\mathbb{R})$, see \cite{Hi} \cite{La} \cite{Lo}. 

In this section, we review some intricate relations between affine spheres and Hilbert metrics.
 We shall also refer to the relation between cubic differentials and affine spheres in \S \ref{s:Hitchin} below. 

Let  $M\subset \R^{n+1}$ be a transversely oriented smooth hypersurface with a trivial bundle $E=M\times \R^{n+1}$. Choose a transverse vector field $\xi$ over $M$ so that $E=T_XM\oplus L$ where $TM$ is the tangent bundle to $M$ and $L$ a trivial line bundle over $M$ spanned by $\xi$. If $\nabla$ is the standard affine flat
connection on $\R^{n+1}$, its restriction on $E$ gives the following equations (Gauss and Weingarten):
\begin{eqnarray}
 \nabla_X Y=D_XY + h(X,Y)\xi,\\
 \nabla_X \xi=-S(X)+ \tau(X)\xi 
 \end{eqnarray}
 for any tangent vector fields $X$ and $Y$ to $M$. The equations are obtained by splitting at each point $x\in M$ the tangent space to $\mathbb{R}^{n+1}$ at that point to $TM_x\oplus L$.
Here, $D$ is a torsion-free connection on $TM$, $h$ a symmetric 2-form on $TM$, $S$ a shape operator and $\tau$ a 1-form. If $M$ is locally strictly convex, i.e., if it can be written locally as the graph of a function with positive definite Hessian, then there exists a unique transverse vector field $\xi$ such that
\begin{enumerate}
\item $\tau=0$, 
\item $h$ is positive definite, and
 \item $|\mathrm {det}(Y_1,\cdots,Y_n,\xi)|=1$ for any $h$-orthonormal frame $Y_i$ of $TM$.
 \end{enumerate}
 (see Proposition 2.1 of \cite{BH}).
 The vector field $\xi$  is called the {\it affine normal},\index{affine normal} $D$ the {\it Blaschke connection},\index{Blaschke connection} and $h$ the {\it affine metric}\index{affine metric}\index{metric!affine} on $M$. These are equiaffine notions.
  \begin{definition} A hypersurface $M$ in $\mathbb{R}^{n+1}$ is said to be an \emph{affine sphere with affine curvature $-1$}\index{affine sphere} if the shape operator satisfies $S=-Id$.
    \end{definition}

 In more geometric terms, an affine sphere in $\mathbb{R}^{n+1}$ is a smooth hypersurface characterized by the fact that its affine normal lines meet at a common point, called the center of the affine sphere (which could be at infinity, and in that case the affine sphere is said to be \emph{improper}). The affine normal field is an affine invariant of the surface and the condition of being an affine sphere is therefore affinely invariant. This makes the family of affine spheres is invariant under the group of affine transformations of $\mathbb{R}^{n+1}$.  Examples of affine spheres include ellipsoids and quadric hypersurfaces. There are two kinds of proper locally convex affine spheres, the \emph{hyperbolic}, all of whose normals point away from the center, and the \emph{elliptic}, all of whose normals point towards the center. An affine sphere is not necessarily affinely equivalent to a Euclidean sphere and, in fact, there are infinitely many non-equivalent affine spheres. See \cite{LS},  \cite{NS} and \cite{LM} for surveys on affine spheres. Another family of objects which is well-known to be invariant under the group of affine transformations is the family of straight lines.

 Affine spheres were introduced at the beginning of the twentieth century by \c Ti\c teica, and they were studied later on by Blaschke, Calabi, Cheng-Yau and others; see the survey by Loftin \cite{LS} for some historical background. The definition of the affine metric using the invariance of the affine normal was derived by Blaschke \cite{B}. 
  
  There are  relations, discovered by Blaschke and by Calabi, between the theory of affine spheres and the real Monge-Amp\`ere equations, and there is also a relation between affine spheres and the theory of convex real projective structures. The work of Cheng-Yau on affine spheres, combined with work of Wang \cite{Wang} on PDEs in the setting of affine differential geometry, was also used by Labourie and Loftin to parametrize equivalence classes of representations in $\mathrm{SL}(3,\mathbb{R})$ by cubic differentials on a Riemann surface (see \S \ref{s:Hitchin} below). 
  
The work of Cheng-Yau on the Monge-Amp\`ere equations associates to each properly convex subset of $\mathbb{R}^{n+1}$ an affine sphere in this space.  Our aim in the rest of this section is to give an idea of how Hilbert geometry fits into this picture, in particular through the following two results:
\begin{itemize}
\item For any proper convex set $\Omega$, the construction of affine spheres by Cheng and Yau leads naturally to a Riemannian metric on $\Omega$ which is bi-Lipshitz equivalent to the Hilbert metric at the level of norms, cf. Proposition \ref{Hilbert-affine} below. 

\item In the case of a strictly convex real projective surface, there is a comparison between the Hilbert volume on that surface with the affine volume (see Corollary \ref{comparison}). 
\end{itemize}

In both cases, the constants that appear in the comparison (that of the norms and that of the volumes) are universal.
  
 We briefly recall the construction by Cheng and Yau of an affine sphere in $\mathbb{R}^{n+1}$ associated to properly convex subset of this space.

A \emph{cone}\index{cone} $\cal C\subset\R^{n+1}$ is a subset which is invariant by the action of the positive reals by homotheties. 
A \emph{convex cone}\index{convex cone} $\cal C\subset\R^{n+1}$ is a cone which is (the closure of) the inverse image of a convex set $\Omega$ in $\mathbb{RP}^n$ by the canonical projection $\mathbb{R}^{n+1}\setminus \{0\} \to \mathbb{RP}^n$. 
 To any bounded open convex subset $\Omega\subset\R^n$, there is an associated convex cone $$\cal C(\Omega)=\{t(1,x)| x\in\Omega, t>0\}\subset \mathbb{R}^{n+1}$$ which is a connected component of the inverse image of a convex set $\Omega$ by the natural map $\mathbb{R}^{n+1}\setminus \{0\} \to \mathbb{RP}^n$ with fiber $\mathbb{R}^*_+$. A group $\Gamma$ of projective transformations which acts properly discontinuously on $\Omega$ also acts on the cone $\mathcal{C}(\Omega)$. This is a consequence of the fact that any representation of a discrete group into $\mathrm{PSL}(n+1,\mathbb{R})$ lifts to a representation into the group $\mathrm{SL}^{\pm}(n+1,\mathbb{R})$ of invertible $(n+1)\times (n+1)$ matrices whose determinant is $\pm 1$. 
 
Cheng and Yau \cite{CY} proved the following, which solved a conjecture of Calabi \cite{Cal}:
\begin{thm}If $\cal C\subset\R^{n+1}$ is an open convex cone containing no lines, then
there exists a unique embedded hyperbolic affine sphere $\mathcal H$ whose center is the origin, which has affine curvature $-1$,
and which is asymptotic to the boundary of $\mathcal C$.
\end{thm}
The fact that $\mathcal{C}$ contains no lines is equivalent to the fact that the convex set $\Omega$ is proper.

The affine invariants of $\mathcal{H}$ are invariants of $\Omega$. 

The projection map induces a homeomorphism between $\mathcal{H}$ and $\Omega$. 
  
 To the bounded open convex set $\Omega\subset\R^n$ is  associated an affine sphere, defined as $$\cal H=\left\{\frac{-1}{u(x)}(1,x)| x\in\Omega\right\}$$ where $u$ is the unique convex solution
of the real Monge-Amp\`ere equation
$$\mathrm{det} D^2 u=(-1/u)^{n+2}$$ satisfying
$$u|_{\partial \Omega}=0.$$
The projection map  $\mathbb{R}^{n+1}\setminus \{0\} \to \mathbb{RP}^n$  induces a diffeomorphism between $\mathcal{H}$ and $\Omega$, and the affine metric $h$ on $\cal H$ induces a Riemannian metric on $\Omega$, still denoted by $h$ and called the \emph{affine metric}.\index{metric!affine}\index{affine metric} This metric gives rise to a measure $\mu_h$ on $\Omega$.  

Now we have two measures on $\Omega$, one is $\mu_h$, coming from the affine metric and the other is $\mu_\Omega$, coming from the Hilbert metric.

Benoist and Hulin proved in \cite{BH} that the affine metric is bi-Lipschitz equivalent to the Hilbert metric:
\begin{Prop}[Benoist-Hulin]\label{Hilbert-affine}There exists a constant $c>0$ such that for any properly convex set $\Omega$,  $x\in\Omega$ and $X\in T_x\Omega$,
$$\frac{1}{c} ||X||_F\leq ||X||_h \leq c ||X||_F,$$ where the subscript $F$ denotes the Finsler norm of the Hilbert metric and $h$ the affine metric.
\end{Prop}
Benoist and Hulin deduce this result from the cocompactness of the action of $\mathrm{SL}^{\pm}(n+1,\R)$ on the set of pairs
$(x,\Omega)$ and the continuous dependence on $(x,\Omega)$ of both affine and Hilbert metrics. The result for Hilbert metrics is contained in Benz\'ecri's thesis, \cite{Ben}, Chapter V, where the author introduces several spaces he calls \emph{body spaces} and \emph{form spaces} (``espaces de corps" and ``espaces de formes"). One of these spaces is, for $n\geq 1$, the space of pairs
$(x,\Omega)$, where $\Omega\subset \mathbb{RP}^n$ is a properly
convex open subset and $x$ is a point in $\Omega$. It is equipped with  the Hausdorff topology. More precisely, the topology on the set of pairs $(x,\Omega)$ is the product topology, where on the first factor the topology is induced by the canonical metric on projective space (which is the quotient metric of the canonical metric on the sphere) and on the second factor the  topology is induced by the Hausdorff  distance on the compact subsets of projective space. (To use precisely the last notion, one needs to replace a subset $\Omega$ by its closure.)
 Benz\'ecri considers then the natural quotient of  this space  by the action of the group $\mathrm{SL}^{\pm}(n+1,\R)$ and he proves that this quotient is 
 compact and metrizable (see  \cite{Ben}, Th\'eor\`eme 2 p. 309).
This implies the following:
\begin{co}[\cite{BH} Proposition 2.6]\label{comparison}There exists a universal constant $C$ (that depends only on the dimension) such that for any convex real projective manifold $M$, we have, for any Borel subset $B$ of $M$, 
$$ \frac{1}{C} \mathrm{Vol}_H(B)\leq \mathrm{Vol}_A(B) \leq C \mathrm{Vol}_H(B),$$ 
where $H$ denotes the Hilbert volume and $A$ the  affine volume.
\end{co}
In particular, we have, in the case where $M$ is a finite volume quotient:
$$
  \frac{1}{C} \mathrm{Vol}_H(M)\leq \mathrm{Vol}_A(M) \leq C \mathrm{Vol}_H(M).
$$ 
Thus, for a convex real projective manifold, having finite affine volume and having finite Hilbert volume are equivalent properties. 

One should mention that in the special case of surfaces and under the additional assumption that the boundary of the convex set is smooth, Loewner and Nirenberg solved the Monge-Amp\`ere equation and they also constructed a Riemannian metric on the convex set which is invariant by projective transformations. By the uniqueness of the solution to the Monge-Amp\`ere equation, this metric is the same as the one constructed by Cheng and Yau using affine spheres.

 The relation between the structure equations of affine spheres and $\mathbb{RP}^2$ structures is explained in detail in the survey \cite{LM} by Loftin and McIntosh.

The asymptotic affine sphere and an invariant Riemannian metric associated to a Hilbert geometry are also mentioned in Chapter 8 of this volume \cite{Marquis-H} (\S 4.1 and 4.2).

We shall say more on affine spheres in \S \ref{s:Hitchin} below.  

\section{Parametrizations of real projective structures and the deformation spaces}\label{para} 
The classification of convex real projective structures on surfaces of nonnegative Euler characteristic is due to Kuiper, cf. \cite{Kuiper1} and \cite{Kuiper}.
In this section, we describe parameter spaces for these structures. We start, in the first subsection,  with Goldman's parametrization, which was inspired by Thurston's exposition of the Fenchel-Nielsen parameterization of Teichm\"uller space associated to a pants decomposition of a hyperbolic surface that consists of the length and the twist parameters of the pants curves.  

The Teichm\"uller space of a surface sits naturally as a subset of the deformation space of real projective structures, and Goldman's parameters are a generalization of the Fenchel-Nielsen parameters. We then describe Hitchin's parametrization, which also generalizes in the parameters for Teichm\"uller space, when this space is considered as a connected component of the character variety of the fundamental group of the surface in $\mathrm{SL}(2,\mathbb{R})$. Hitchin's parameter space is a component of the character variety of representations of the fundamental group of the surface in $\mathrm{SL}(n,\mathbb{R})$. Hitchin's work gave rise to several other works by various authors, and it is also related to several questions addressed in this paper. We shall mention some of these relationships below. In the final subsection, we mention a set of parameters due to the first author of this paper that use the length spectra of the Hilbert metrics. This is also a generalization of a classical parametrization of hyperbolic structures by geodesic length spectra. 

It is an interesting question to study more carefully the structure of the parameter spaces and their nature (analytic, algebraic, etc.)

\subsection{Goldman's parametrization}\label{s:Gol}

Our goal in this subsection is to give a brief description of Goldman's parameters for the deformation space of convex real projective structures on surfaces of nonnegative Euler characteristic. 

Recall that by using the Cayley-Klein-Beltrami model of hyperbolic geometry, a hyperbolic structure (in the sense of a Riemannian metric of constant curvature equal to -1) on a closed surface $S$ of negative Euler characteristic is a special case of a convex real projective structure. This is a structure of the type $\Omega /\Gamma$ where $\Omega\subset \mathbb{R}^n\subset \mathbb{RP}^n$ is the interior of an ellipse and $\Gamma$ a subgroup of the projective transformations of $\mathbb{RP}^n$ that preserve $\Omega$.

 For any topological surface $S$ of negative Euler characteristic, a classical and useful set of parameters for the space of isotopy classes of hyperbolic structures on $S$ (that is, for the Teichm\"uller space of $S$) is provided by the Fenchel-Nielsen coordinates associated to pair of pants decompositions. These parameters consist of the set of lengths of the pants curves of this decomposition, rendered geodesic for  the hyperbolic structure, together with the twist parameters along these curves that measure the way in which the pairs of pants are glued together. (One needs a convention to measure the twists, whereas the length parameters are intrinsic.) In the case of a closed surface of genus $g\geq 2$, an Euler characteristic argument shows that the number of curves in any pants decomposition of $S$ is $3g-3$. Thus, the length and twist parameters of $S$ make a total of $6g-6$ parameters, which is indeed the dimension of the Teichm\"uller space of $S$. These parameters are the so-called Fenchel-Nielsen parameters.\index{Fenchel-Nielsen parameters}

For convex real projective structures on surfaces, there are analogues of the Fenchel-Nielsen parameters associated to a pair of pants decomposition, and we now describe them briefly.  This parametrization of the space of equivalence classes (the so-called deformation space)\index{convex projective structure!deformation space} of convex real projective structures on a closed surface of genus $g\geq 2$ was obtained by Goldman \cite{Goldman},\index{convex projective structure!Goldman parameters}\index{Goldman parameters} who showed that this deformation space, $\mathcal{D}(S)$, is homeomorphic to an open cell of dimension $16(g-1)$. In this projective setting, the pants curves are made geodesic with respect to the Hilbert metric (in the coordinate charts, the curves are made affine straight lines). The length of a simple closed geodesic in a given homotopy class is uniquely defined, and this set of lengths is part of the parameters associated to pair of pants decomposition. However, in the setting of projective structures, the other parameters of the Fenchel-Nielsen coordinates are more complicated to describe that in the case of Teichm\"uller space; they are described in terms of $3\times 3$ matrices that represent the curves on the surface. In what follows, we shall give an idea of these parameters.

We recall that the fundamental group of a surface equipped with a convex real projective structure acts freely and properly discontinuously on the convex set $\Omega$ which is the image of the associated developing map. Thus, instead of talking about parameters for the equivalence classes of convex real projective structures on a given closed surface $S$, one can talk about parameters for the equivalence classes of properly convex open subsets of $\mathbb{RP}^2$ equipped with a properly discontinuous free action of a group isomorphic to the fundamental group $\pi_1(S)$. 
In this way, the deformation space can be viewed as an open subset of the character variety\index{character variety} 
\[
\mathrm{Hom}(\pi_1(S)\to \mathrm{PSL}(3,\mathbb{R}))/\mathrm{PSL}(3,\mathbb{R})
\]
of representations of the fundamental group $\pi_1(S)$ into the Lie group $\mathrm{PSL}(3,\mathbb{R})$ and where the quotient is by the action of $\mathrm{PSL}(3,\mathbb{R})$ on the space of representations by conjugation.
The Teichm\"uller space $\mathcal{T}(S)$ becomes a subspace of $\mathcal{D}(S)$. In the general case of a compact surface $S$ of finite type  with $n$ boundary components with negative Euler characteristic $\chi(S)$, denoting by  $\mathcal{D}(S)$  the deformation space of convex real projective structures on $S$,  Goldman proves the following (see \cite{Goldman} Theorem 1):
\begin{thm}\label{thG}
The space  $\mathcal{D}(S)$ is an open cell of dimension $-8\chi(S)$, and the map which associated to each convex projective surface $S$ the germ of its structure near $\partial S$ is a fibration of $\mathcal{D}(S)$ over an open $2n$-cell with fiber an open cell of dimension $-8\chi(S)-2n$.
\end{thm} 

We shall explain below the behavior of the projective structures near the boundary.

In the rest of this subsection, we give a brief description of Goldman's parameters.

 The parameters are based on pairs of pants decompositions of the surface. In what follows, we shall assume that the pairs of pants decompostions that we consider are all geodesic. The parametrization of convex real projective structures on $S_g$ is then done by first describing the structures on individual pairs of pants and then gluing, as we usually do in Teichm\"uller theory for the parametrization of hyperbolic structures. It turns out however that in the case of general real projective structures, the parameters associated to the pair of pants and to their gluing are more intricate than in the case of hyperbolic structures. 
  
In this way, one is led to the question of understanding convex projective structures on pairs of pants, and this naturally requires the study of convex structures on surfaces with boundary.  

Recall that a convex projective structure on a surface $S$ is a maximal atlas with values in $\mathbb{RP}^2$ whose transition functions are restrictions of projective automorphisms of $\mathbb{RP}^n$. We need to consider convex projective structure on surfaces with boundary. This is also defined by an atlas with values in $\mathbb{RP}^2$, and we furthermore require that for each open set $U$ in $S$ which is the domain of a chart satisfying $U\cap \partial S\not=\emptyset$ and for each arc in $U$ which is in $\partial S$, the image of this arc by the chart map is contained in a projective line of $\mathbb{RP}^2$. The boundary components therefore become \emph{closed geodesics}, and we require that each such closed geodesic has a geodesically convex collar neighborhood whose holonomy has distinct positive eigenvalues. Such structures are sometimes called structures ``with standard convex projective collar neighborhood". The image by the holonomy map of a loop representing a boundary component, which is well defined up to conjugacy, is a matrix in $\mathrm{SL}(3,\mathbb{R})$ which has three distinct and positive real eigenvalues whose product is equal to 1. Such a matrix is termed by Goldman \emph{positive hyperbolic}.\index{positive hyperbolic!transformation} The matrix in $\mathrm{SL}(3,\mathbb{R})$ associated to such a boundary component  is well defined up to conjugacy and it plays the role of the Fenchel-Nielsen length parameter in the hyperbolic setting. Thus, we have a \emph{Fenchel-Nielsen type length parameter} associated to a closed geodesic  which is  two-dimensional in the case of convex projective structures. One can glue projective structures on surfaces with boundary with standard convex projective collar neighborhoods by identifying the collar neighborhoods of boundary components using projective isomorphisms. Goldman proves in \cite{Goldman} that if we glue in this way a finite number of convex projective structures on surfaces with boundary which all have negative Euler characteristic, then the resulting projective structure is convex.

On each pair of pants, a convex real projective structure is determined by $3\cdot 2 +2$ parameters. Here, 3 is the number of boundary curves of the pair of pants, and near each boundary curve,
the projective structure  is determined by the two real parameters which we mentioned above. Also, whereas a hyperbolic structure on a pair of pants is completely determined by the ``length parameters" associated to the boundary components (which are the lengths of these bondary components, including boundary component of length zero, which correspond to cusps),  a general convex real projective structure on a pair of pants is not determined by the sole length parameters associated to the boundary, but there are two more extra real parameters involved, called ``twist parameters".

To be more precise, recall first that the projective transformations of $\mathbb{RP}^2$ can be represented by $3\times 3$ matrices. These matrices act on the vector space $\mathbb{R}^3$, and their action on $\mathbb{RP}^2$  is the quotient action. The whole discussion can be reduced then to considerations on matrices and their actions on $\mathbb{R}^{3}$, or, more precisely, on the set of lines through the origin of that vector space (which is the projective space $\mathbb{RP}^{2}$). Note also that the group acting by matrices on $\mathbb{R}^{3}$ is in fact the group $\mathrm{SL}^{\pm}(3,\mathbb{R})$ and not $\mathrm{PSL}(3,\mathbb{R})$, but we shall not worry here about the difference between the two groups.

Thus, the holonomy of each closed geodesic on the surface $S$ is positive hyperbolic. In particular, the set of parameters for   $\mathrm{SL}(3,\mathbb{R})$-conjugacy classes of such a matrix is an open 2-cell. Goldman, in \cite{Goldman}, gives three equivalent sets of coordinates for the space of conjugacy classes of such matrices, and he makes a complete study of the dynamics of the action of a transformation of $\mathbb{RP}^2$ representing such an element. We are now interested in the automorphisms of $\mathbb{RP}^2$ that preserve $\Omega$. For a hyperbolic element of  $\mathrm{Aut}(\Omega)$, the attractive and repulsive fixed points belong to the boundary of the convex set $\Omega$. A lift to $\Omega$ of the geodesic in the quotient surface $S$ is preserved by the action of the corresponding affine transformation, and that action is indeed hyperbolic in the sense that it is a translation along that geodesic, with an attractive and a repelling fixed point as endpoints. 

A positive hyperbolic isometry representing a boundary component of a pair of pants in the decomposition has three distinct eigenvalues $\lambda_1>\lambda_2>\lambda_3$ and it is conjugate to a diagonal matrix with these eigenvalues, $\lambda_1,\lambda_2,\lambda_3$, in that order, on the diagonal. The action on $\mathbb{RP}^2$ of such a matrix has three fixed points. Using the coordinates of $\mathbb{R}^3$, these points correspond to the line passing through $(1,0,0)$, which is an attracting fixed point, the line passing through $(0,1,0)$, which is a saddle point, and the line passing through $(0,0,1)$, which is repelling. Now the three lines in $\mathbb{RP}^2$  passing through the pairs of such points divide the space into four triangles which are invariant by the action of the given diagonal matrix. Goldman in \cite{Goldman} describes in detail the action of this matrix on these triangles. 

In the above matrix representation, the first parameter for the projective structure on the pair of pants which is associated to a closed curve which is the boundary of a pair of pants is $\log \frac{\lambda_1}{\lambda_3}$, which is the Hilbert metric length of the closed geodesic. The other parameter is $3\log \lambda_2$.  (We are following Goldman's exposition,  and the factor 2 is a convenient normalization.) These are the 6 parameters associated to the boundary curves of a pair on pants.  The extra two parameters are called by Goldman ``interior parameters".
In conclusion, we have the following proposition:
\begin{Prop}[Goldman \cite{Goldman}] The deformation space of $\mathbb{RP}^2$ structures on a pair of pants is an open cell of dimension 8.
\end{Prop}
When one glues two convex projective structures on pairs of pants along simple closed curves, there are two new parameters involved, associated to a simple closed curve $C$ obtained after gluing. One of these parameters is called the twisting parameter, and the other one is called the vertical twist parameter. The two parameters combined are the analogues of the single Fenchel-Nielsen twist parameter of the case of hyperbolic surfaces. In fact, in the case where the projective structure is a hyperbolic structure, then the twisting parameter is the usual twist parameter. More precisely, given a diagonal hyperbolic isometry $(\lambda_1,\lambda_2,\lambda_3)$ representing $C$, the matrices
$$\gamma_{t}=\left [\begin{matrix}
  e^{t} & 0 & 0 \\
  0   & 1 & 0 \\
  0   & 0 & e^{-t}
  \end{matrix}\right], O_t=\left[\begin{matrix}
                   e^{-\frac{1}{3}t} & 0 & 0 \\
                   0 & e^{\frac{2}{3}t} & 0 \\
                   0 & 0 & e^{-\frac{1}{3}t}\end{matrix}\right] $$  commute with  the diagonal matrix $(\lambda_1,\lambda_2,\lambda_3)$ where $\lambda_1,\lambda_2,\lambda_3$ are the three eignevalues which we described above and the parameters associated to these two matrices are used to glue two structures along $C$. Here, $\gamma_t$ is the twisting parameter and $O_t$ is the vertical parameter. Suppose that two pairs of pants $P_1$ and $P_2$ are glued along $C$. Once a projective structure on $P_1$ is given, the two parameters corresponding to the length and $\lambda_2$ for $C$ are already determined for $P_2$, and the twisting and the vertical parameters are needed to glue along $C$. In any case, the total parameters are $8(2g-2)$, since there are $2g-2$ pairs of pants and  8 parameters on each pair of pants. 

The vertical parameter gives rise to a so-called bulging deformation\index{bulging deformation} of the projective structure. This deformation is associated to a measured geodesic lamination on a hyperbolic surface, and it deforms the underlying $\mathbb{RP}^2$ structure of that surface. The bulging deformation is an extension of the earthquake deformation\index{earthquake deformation} (which itself is the extension of the Fenchel-Nielsen deformation from closed geodesics to measured geodesic laminations) and of the bending deformation in $\mathrm{PSL}(2,\mathbb{C})$ (which is a complexification of the Fenchel-Nielsen deformation) to the context of convex real projective structures. The bulging deformation was introduced by Goldman in \cite{Goldman}, who also wrote a recent paper on that subject \cite{Go}.

 Let us recall this deformation.
For a closed
hyperbolic surface $S$ and a closed geodesic $\gamma$ on $S$, after conjugation, we may assume that
$$\rho(\gamma)=\gamma_{t_0}=\left [\begin{matrix}
  e^{t_0} & 0 & 0 \\
  0   & 1 & 0 \\
  0   & 0 & e^{-t_0}
  \end{matrix}\right],\ t_0>0$$ as an element of $\mathrm{SL}(3,\R)$. Now let $\gamma_t$ be the one-parameter subgroup generated by
  $\rho(\gamma)$. If $\Omega$ is an ellipse so
that $S=\Omega/\rho(\pi_1(S))$, the eigenspaces $\R v_1,\R v_2,\R
v_3$ corresponding to $e^t,1,e^{-t}$ respectively, define three
points $\gamma_+, \gamma_0, \gamma_-$ in $\mathbb{RP}^2$ where $\gamma_\pm$
are the attracting and the repelling fixed points of $\rho(\gamma)$ on
$\partial \Omega$, and $\gamma_0$ is outside of $\Omega$. Consider a
triangle $\triangle$ passing through these three points with left
and right vertices corresponding to $\gamma_+,\gamma_-$, and top vertex to $\gamma_0$. Then the dynamics of $\rho(\gamma)$ on $\triangle$ is
from the right vertex to the left and top vertices, and from the top
vertex to the left vertex. Inside $\triangle$, the orbits of
$\gamma_t$ are arcs of conics tangent to $\triangle$, one of
which is the segment $C$ of $\partial\Omega$ from $\gamma_-$ to
$\gamma_+$. See the picture in \cite{Goldman}.

The time-$t$ earthquake map of $S$ along $\gamma$ is given by the partial right Dehn
twist along $\gamma$, which amounts to moving the right hand side of the
lifts of $\gamma$ by the amount $t$ in $\Omega$. This can be
realized by conjugating the action of the right hand side of $S\setminus\gamma=S_1\cup S_2$ (if
$\gamma$ is separating) by $\gamma_t$, $\rho(\pi_1(S_1))*_{\langle \rho(\gamma) \rangle} \gamma_t \rho(\pi_1(S_2))\gamma_t^{-1}$, and correspondingly, using an HNN
extension for the non-separating case.

Obviously this earthquake deformation does not change the domain
$\Omega$. To deform the domain, we perform the bulging deformation which we already mentioned.
We want to replace $C$ by another conic which is an orbit of
$\gamma_t$ tangent to $\triangle$. This can be realized by
conjugating the right hand side of $S\setminus\gamma$ (if $\gamma$ is separating) by $$O_t=\left[\begin{matrix}
                   e^{-\frac{1}{3}t} & 0 & 0 \\
                   0 & e^{\frac{2}{3}t} & 0 \\
                   0 & 0 & e^{-\frac{1}{3}t}\end{matrix}\right], $$ and correspondingly
using an HNN extension for the non-separating case. What $O_t$ does to the
domain is stretching it in the $\gamma_0$ direction, which entails to
move the boundary arc $C$ to one of the conics traced by $\gamma_t$
outside $\Omega$ if $t>0$, to one inside $\Omega$ if $t<0$.

When $t\ra\pm\infty$, the domain $\Omega$ degenerates to the one containing edges of $\triangle$. By Theorem \ref{finitevolume}, the degenerate
structures have infinite Hilbert area. Geometrically, the degenerate structure has an infinitely long cylinder attached along $\gamma$.
The following is an open question:

\smallskip

{\bf Question.} Is the topological entropy function decreasing to some non-zero number  in the above bulging deformation?

\subsection{Hitchin's parametrization}\label{s:Hitchin}
Hitchin parametrized a specific component, called the Hitchin component,\index{Hitchin component} of the character variety of representations of
$\pi_1(S)$ in $\mathrm{SL}(n+1,\br)$ \cite{Hi}. This is the component that contains the representations of hyperbolic structures, and it has some properties which are analogous to those of the Teichm\"uller space, which is a component of the character variety of representations in $\mathrm{SL}(2,\br)$.

In the case $n=2$ and if a complex structure on $S$ is fixed,  Hitchin identified the Hitchin component with the vector space $H^0(S, K^2\oplus K^3)$ of holomorphic quadratic and cubic differentials over the Riemann surface $S$. Here $K$ is the canonical line bundle of $S$. In his work, Hitchin used the techniques of Higgs bundles, which are holomorphic vector bundles equipped with so-called ``Higgs fields" that appeared in earlier works of Hitchin and of Simpson in their study of Teichm\"uller space.

The Hitchin component for $n=2$ (that is, representations in $\mathrm{SL}(3,\br)$)) was  described  again by
Labourie \cite{La} and Loftin \cite{Lo} independently using affine spheres, a notion which we considered in Section \ref{affine}. Both Labourie and Loftin showed that the space of equivalence classes of convex real projective structures on a closed oriented surface of genus $\geq 2$ is parametrized by the space of  conformal structures equipped with holomorphic cubic differentials. Using the theorem of Riemann-Roch gives another proof of the result of Goldman that this space of equivalence classes is an open cell of dimension $16(g-1)$ (cf. Theorem 4.3). The parametrization is by a fiber bundle whose   base (Teichm\"uller space) has dimension $3g-3$ and each fiber has dimension $5g-5$. This parametrization has also the advantage of equipping the deformation space of convex real projective structure with a natural complex structure.
Labourie made in \cite{La} the relation between these cubic differentials and those which appears in Hitchin's parametrization.

The proof of the correspondence established by Labourie and Loftin is based on the works of Cheng-Yau \cite{CY} that we mentioned above and the work of Wang \cite{Wang}. It starts with the fact that $S$ can be written as the quotient $\Omega /\Gamma$ where $\Omega$ is a bounded convex domain in $\mathbb{R}^2$ and $\Gamma$ a subgroup of $\mathrm{SL}(3,\mathbb{R})$ acting properly discontinuously on $\Omega$,  and it makes use of the fact that $\Omega$ can be canonically identified with the affine sphere $H$ asymptotic to the open cone $\mathcal{C}\subset \mathbb{R}^3$ that sits over it, cf. \S \ref{affine}. The group $\Gamma$ can be lifted to a group acting linearly on $\mathbb{R}^{n+1}$ and preserving the affine sphere. The natural projection $\mathcal{C}\to \Omega$ induces a diffeomorphism between $H$ and $\Omega$. The affine metric associated to the affine sphere $H$ induces a Riemann surface structure on $\Omega /\Gamma$, and the cubic differential on this Riemann surface is essentially obtained by taking the difference between the Levi-Civita connection of the affine metric on $H$ and the Blaschke connection of $H$. This follows from the work of Wang \cite{Wang} who relates convex projective structures on a surface to holomorphic data. Wang worked in the setting of affine differential  geometry. He gave a condition in terms of the affine metric for a two-dimensional surface to be an affine sphere that involves conformal geometry.  See also Labourie \cite{La} and Loftin \cite{Lo1}.

The work of Labourie and Loftin has been extended to the case of noncompact surfaces of finite Hilbert volume by Benoist and Hulin \cite{BH}. 

For the relation between cubic differentials and the differential geometry of surfaces, we refer the reader to \cite{LM}.

\subsection{Length spectra as parameters} \label{s:leng}

Let $M$ be a manifold equipped with a metric $g$ and let $\mathcal{S}$ be the set of free homotopy classes of simple closed curves on $S$. The \emph{marked length spectrum}\index{marked length spectrum} of $(M,g)$ is the function on $\mathcal{S}$ which associates to each element the infimum of the lengths of closed curves in that free homotopy class. In the above definition, the adjective \emph{marked} refers to the fact that we are not only considering the set of lengths associated to the elements of $\mathcal{S}$, but we keep track of each element of $\mathcal{S}$ with its associated length parameter. In what follows, all length spectra that we consider are marked, and we shall sometimes denote marked the length spectrum by the term \emph{length spectrum}.\index{length spectrum}

The first author of this paper showed in \cite{kim1} that a normalized version of the length spectrum\index{length spectrum} for the Hilbert metric is a set of parameters for convex projective structures on surfaces, up to dual structures. Here a dual structure\index{convex structure!dual}\index{dual convex structure} is the real projective structure induced from the dual cone $\Omega^*=\{w: \langle w, v \rangle > 0, \forall v\in \overline{\Omega}-\{0\}\}$. (Note that length spectra cannot distinguish between dual structures.) This result is an analogue of the fact that the projectivization of length spectra of hyperbolic structures on a surface of finite type can be used as parameters for the Teichm\"uller space of that surface. Thurston used these parameters on Teichm\"uller space in his construction of his boundary whose elements are projective equivalence classes of measured foliations (or, equivalently, of measured laminations), see \cite{Thurston-b} and \cite{FLP}. The analogous result for convex projective structures on the surface was also used in order to define a boundary for the deformation space of these convex structures. In fact, Kim's parameters are expressed in terms of logarithms of eigenvalues of hyperbolic $3\times 3$ matrices; see \cite{kim1} for the details. We shall get back to this matter in \S \ref{geodesiccurrents}.

\section{Strictly convex manifolds and topological entropy}\label{entropy}
We already mentioned that strictly convex real projective structures have several properties which are similar to properties of negatively curved Riemannian manifolds. Let us give a few examples.

Benoist \cite{Be} proved the following fundamental theorem about strictly convex projective manifolds, relating the regularity of the boundary of the universal covering $\Omega$ to the large-scale geometry of a group that divides it.
\begin{thm}\label{th:Gromov}
Let $\Omega$ be a divisible convex  set, divided by a group $\Gamma$. Then the following are equivalent:
\begin{enumerate}
\item $(\Omega,d_\Omega)$ is Gromov-hyperbolic;
\item $\Omega$ is strictly convex;
\item $\partial \Omega$ is $C^1$;
\item The geodesic flow on $\Omega/\Gamma$ is Anosov.
\end{enumerate}
\end{thm}

Let us make a few comments on these statements. 

Although the Hilbert metric is generally non-Riemannian, there is a natural notion of geodesic flow\index{geodesic flow} associated to the projective manifold $\Omega/\Gamma$, which is the quotient by $\Gamma$ of the geodesic flow on $\Omega$. The geodesic flow of  $\Omega$ is defined on the unit tangent bundle $T_1(\Omega)\simeq  (T\Omega\setminus\{0\})/\R^*_+$. This is the space of pairs $(x,\xi)$ where $x$ is a point in $\Omega$ and $\xi$ is a vector of length one based at $x$ (remember that for strictly convex manifolds, the projective lines are exactly the Euclidean lines) with respect to the Hilbert metric. The geodesic flow $\phi^t:T_1(\Omega) \to T_1(\Omega)$ is then obtained by moving the pair $(x,\xi)$ unit length in the direction specified by $\xi$. The quotient flow, denoted by the same name, $\phi^t:T_1(\Omega/\Gamma) \to T_1(\Omega/\Gamma)$, is the geodesic flow associated to $\Omega/\Gamma$.

We recall that a $C^1$ flow $f^t:W\to W$ on a Riemannian manifold $W$ is said to be Anosov\index{Anosov flow}\index{flow!Anosov} if there is a splitting of the tangent space at each point $v$ into three subspaces  
\[T_vW=E^u(v)\oplus E^s(v)\oplus \mathbb{R}^*(X)\]
where $E^u, E^s$ have positive dimension, $E^u$ expanding under the flow, $E^s$ contracting and $X$ the generator of the flow. A typical Anosov flow is the geodesic flow on the unit tangent bundle of a manifold of negative curvature. Anosov flows have interesting dynamical properties, for instance, the union of periodic orbits are dense. The definition of an Anosov flow can be transcribed in a straightforward way to the setting of the Finsler manifolds $\Omega$ and $\Omega/\Gamma$.

We also recall that Gromov hyperbolicity is a property of the large-scale geometry of a metric space.\index{hyperbolic metric space}\index{metric!Gromov hyperbolic} In the case of a geodesic metric space\index{geodesic metric space} (that is, a metric space in which distances between points are realized by length of curves joining them), the property says that triangles are uniformly $\delta$-thin,\index{triangle!$\delta$-thin} that is, that there exists a $\delta>0$ such that for any triangle, any side is contained in the $\delta$-neighborhood of the two others. See \cite{Gromov} and \cite{CDP}.

It follows from Theorem \ref{th:Gromov} that convex real projective structures on a given closed manifold $M$ are either all strictly convex or they are all non-strictly convex.

The \emph{topological entropy}\index{topological entropy}\index{entropy!topological} of a flow $\phi^t:W\ra W$ on a compact manifold $W$ with distance function $d$ is defined as follows.

For all $t\geq 0$, let $d_t$ be the function on $W\times W$ defined by
$$d_t(x,y)=\max_{0\leq s\leq t} d(\phi^s(x),\phi^s(y))$$
for any two points $x$ and $y$ in $W$.
It is not difficult to see that $d_t$ satisfies the properties of a distance function.

 For any $\epsilon>0$ and $t\in\R$, we consider coverings of $W$ by open sets of diameter less than $\epsilon$ with respect to the distance $d_t$ and we let $N(\phi,t,\epsilon)$ be the minimal cardinality of such a covering. 
 
 The topological entropy of $\phi$ is then defined as
$$h_{\mathrm{top}}(\phi)=\lim_{\epsilon\ra 0}\big(\limsup_{t\ra\infty}\frac{1}{t}\log N(\phi,t,\epsilon)\big).$$

For a strictly convex compact manifold $M$, one can consider the associated geodesic flow on the unit tangent bundle 
$T_1M$ of $M$. The strict convexity ensures that in the local projective charts, any geodesic is contained in a projective line, and therefore there are unique geodesics between any two given points.

Benoist started the study of this flow in \cite{Be} and he showed that it is Anosov and topologically mixing, generalizing properties which were known  to hold for the geodesic flow associated to a hyperbolic manifold. We recall that a flow $f^t$ on a space $M$ is said to be \emph{topologically mixing}\index{flow!topologically mixing}\index{topologically mixing flow} if for any two open sets $A$ and $B$ in $M$, there exists a real number $t_0$ such that for every $t>t_0$, we have $f^t(A)\cap B\not=\emptyset$. An Anosov flow is not necessarily topologically mixing and vice versa.

There are several interesting global questions concerning the long-term behaviour of orbits and more generally the dynamics and the ergodic theory of Anosov flows. Some of these questions are considered by Crampon \cite{crampon} who continued the study initiated by Benoist. 
Crampon showed the following result on entropy:
\begin{thm}Let $\phi^t$ be the geodesic flow of the Hilbert metric on a strictly convex projective compact manifold $M$ of dimension $n$. Its topological entropy satisfies
$$ h_{\mathrm{top}}(\phi^t)\leq (n-1),$$ with equality if and only if $M$ is a hyperbolic manifold with constant curvature $-1$.
\end{thm}
Some of the results of Crampon are developed in detail in Chapter 7 of this volume \cite{Cr}. 

The regularity of the boundary of the convex set plays an important role in the study of the geodesic flow.

\section{Geodesic currents on strictly convex surfaces}\label{geodesiccurrents}

We recalled in \S \ref{s:leng} that it was proved in \cite{kim1} that if two strictly convex compact real projective manifolds have the same length spectrum with respect to the Hilbert metric, then they are projectively equivalent, up to dual structures. Thus, the length spectrum can be used as a parameter space for real projective structures up to dual structures. The length spectrum, and more generally, the set of geodesics in the manifold, or the space of geodesics in the universal cover equipped with the action of the fundamental group, can be studied from various points of view: dynamical, measure theoretical, etc. We shall consider more particularly the case where the manifold is two-dimensional. We start by recalling the notion of geodesic current.

Let $S=\Omega/\Gamma$ be a closed surface of genus at least two equipped with a strictly convex real projective structure.
A {\it  geodesic current}\index{geodesic current}\index{current} on $S$ is a $\Gamma=\pi_1(S)$-invariant Borel measure on the set 
\[(\partial\Omega\times \partial\Omega\setminus
\mathrm{Diag})/\mathbb{Z}_2\] where $\mathrm{Diag}$ is the diagonal set and where $\mathbb{Z}_2$ acts by interchanging the coordinates. Equivalently, a geodesic current is an invariant transverse measure for the geodesic flow on the unit tangent bundle of $S$. The unit tangent bundle $T_1 S$ of $S$ can be seen as the quotient of the set of bi-infinite geodesics $\Omega$ by the action of $\Gamma$. The term ``geodesic" denotes here the lines in the sense of projective geometry, but we also recall that in the case where $\Omega$ is strictly convex (and in fact, it suffices that $\partial \Omega$ does not contain two affinely independent non-empty open segments), the set of projective geodesics coincides with the set of geodesics (in the sense of distance-minimizing curves) for the associated Hilbert metric. The equivalence between the  above two definitions of a geodesic current is based on the fact that from any two distinct points in $\partial \Omega$ there is a unique projective geodesic in $\Omega$ having these points as endpoints. The invariant transverse measures for the geodesic flow of $S$ are also the $\Gamma$-invariant invariant transverse measures of the geodesic flow on $\Omega$. 

There are methods for obtaining invariant measures for geodesic flows and we shall mention the Bowen-Margulis measure below. In the setting of Riemannian manifolds of negative curvature, there are classical methods for constructing transverse measures for geodesic flows, and the theory of such transverse measures is well developed. The methods have been adapted by several authors to the case of Hilbert geometry. For more information on this subject, we refer the reader to the chapter by Crampon in this volume \cite{Cr}. 
 
 Let $\rho:\pi_1(S)\ra \mathrm{SL}(3,\R)$ be a holonomy representation associated to a real projective structure on $S$. We now explain a method of constructing geodesic currents which originates in the work of Ledrappier \cite{Le}. 
 
 We start with the fact that for a strictly convex projective structure, the boundary $\partial \Omega$ is $C^1$, and we recall the definition of a Busemann  cocycle $B$ in this context, cf. \cite{kim1}. 

For a fixed base point $o\in\Omega$ and for $\xi\in \partial \Omega$, choose a unit speed geodesic ray $r(t)$ with $r(0)=o$ and $r(\infty)=\xi$, and let
 $$B_\xi(o, y)=\lim_{t\ra\infty}\left( d_\Omega(y,r(t))-t\right).$$
 The map $t\mapsto d_\Omega(y,r(t))-t$ is non-increasing and bounded from below, therefore the above limit exists. This function is called the \emph{Busemann function}\index{Busemann function} associated to the geodesic ray $r$. Since the geodesic ray between a point in $\Omega$ and a point in $\partial \Omega$ is unique, one can think of the Busemann function as a function on $\partial \Omega$ that depends on the choice of a basepoint in $\Omega$. The function changes by an additive constant when we change the basepoint. When $\Omega$ is the unit disc (i.e. the hyperbolic space), the value of $B_\xi(o, y)$ is the signed distance between the two horospheres based at $\xi$ and passing by $o$ and by $y$. One also talks about the \emph{Busemann cocycle}\index{Busemann cocycle} $B_\xi(o, r^{-1}o)$ associated to a point $\xi\in \partial \Omega$; the reason for this terminology  is the cocycle property expressed in Equation (\ref{cocycle}) below.

 We refer the reader to \cite{Pap} for a systematic presentation of several properties of Busemann functions.

  The group $\Gamma$, being isomorphic to the fundamental group of a closed surface of genus $\geq 2$, is hyperbolic in the sense of Gromov (meaning that its Cayley graph with respect to some -- or, equivalently to any -- finite generating system, equipped with the word metric, is a hyperbolic geodesic metric space). Such a group has a well-defined Gromov boundary. Such a boundary  carries a natural H\"older structure, see \cite{Gromov} and \cite{CDP}.  We shall also use the canonical identification between the two boundary spaces $\partial\Omega$ and $\partial \Gamma$.
 Since the surface $S$ is closed, each element $\gamma\in\pi_1(S)\simeq \Gamma$ is hyperbolic, that is, it acts  on $\Omega$ without fixed point leaving invariant a unique geodesic of that space and acting as a translation along that geodesic, with one attractive and one repelling fixed point in $\partial \Omega$, the two endpoints of the invariant geodesic.
 The \emph{translation length}\index{translation length}\index{length!translation} of $\gamma$, denoted by $\ell(\gamma)$, is defined as 
 \[\ell(\gamma)=\inf_{x\in \Omega} d_{\Omega} (x,\gamma(x)).\]
 Checking the translation length of a group element is generally a pleasant exercise.

 The attracting fixed point at infinity of $\gamma$ is denoted by $\gamma^+$.

  For each $\xi\in\partial \Omega$, consider a Busemann cocycle
$$(\gamma, \xi)\mapsto B_\xi(o,\gamma^{-1}o),$$
 where $o\in\Omega$ is a fixed base point.
Then
$$B_\xi(o,\gamma_1^{-1}\gamma_0^{-1}o)=B_{\gamma_1 \xi}(\gamma_1 o,\gamma_0^{-1} o)=B_{\gamma_1 \xi}(o,\gamma_0^{-1}o)+B_{\gamma_1 \xi}(\gamma_1 o,o)$$$$=B_{\gamma_1 \xi}(o,\gamma_0^{-1}o)+B_\xi(o,\gamma_1^{-1}o).$$
Hence, if we set $c(\gamma,\xi)=B_\xi(o,\gamma^{-1}o)$, we have
\begin{equation}\label{cocycle}c(\gamma_0\gamma_1,\xi)=c(\gamma_0,\gamma_1\xi)+c(\gamma_1,\xi).
\end{equation}
This map $c:\Gamma\times \partial\Gamma\ra \R$ is a H\"older cocycle,\index{H\"older cocycle} that is, besides the cocycle property given by the preceding equation, $c(\gamma,\cdot)$ is a H\"older map for every $\gamma\in\Gamma$.
The \emph{period}\index{period!cocycle} of $c$ at $\gamma$ is defined to be
$$\ell_c(\gamma)=c(\gamma,\gamma^+).$$ 
The reason for this terminology is that two H\"older cocycles are cohomologous if and only if they have the same periods (see Theorem 1.a of \cite{Le}).

In our case,
$$\ell_c(\gamma)=B_{\gamma^+}(o,\gamma^{-1}o)=\ell(\gamma)=\ell(\gamma^{-1})=\ell_c(\gamma^{-1}).$$ Hence the set of periods of $c$ is just the length spectrum of the real projective structure with respect to the Hilbert metric. Such a cocycle $c$ (that is, a cocycle satisfying the property  $\ell_c(\gamma)=\ell_c(\gamma^{-1})$),  is said to be even.

The exponential growth rate of a H\"older cocycle $c$ is defined as
$$h_c=\limsup_{s\ra\infty} \frac{\log \#\{[\gamma]\in [\Gamma]:\ell_c(\gamma)\leq s\}}{s}$$ where $[\gamma]$ denotes the conjugacy class of $\gamma$.
In \cite{Le}, it is shown that if $0<h_c<\infty$ then there exists an associated Patterson-Sullivan\index{Patterson-Sullivan meaure}\index{measure!Patterson-Sullivan} measure on $\partial \Gamma$,  i.e., a probability measure
$\mu$ on $\partial \Gamma$ such that
$$\frac{d\gamma_*\mu}{d\mu}(\xi)=e^{-h_c c(\gamma^{-1},\xi)}.$$
In the case at hand, $h_c$ is just the growth rate of lengths of closed geodesics in the Hilbert metric.

We recalled the definition of the topological entropy of a flow in \S \ref{entropy}. Now we need the notion of \emph{volume entropy}\index{entropy!volume}\index{volume entropy} of a metric. This is a measure of the asymptotic growth rate of volumes of metric balls. We recall that the Hilbert metric on the convex set $\Omega$ equips this set with a notion of volume, viz., the Hausdorff measure associated to the metric. It is called the Hilbert volume (see \S \ref{s:preliminaries}). This notion of volume descends to the quotients of $\Omega$ by properly discontinuous actions of groups of projective transformations. We choose a point $x$ in $\Omega$. The \emph{volume entropy}, $h_{\mathrm{vol}}$ of the metric is then defined as
\[
h_{\mathrm{vol}} = \lim_{r\to\infty}\frac{1}{r} \log \mathrm{vol}(B(x,r))
\]
if this limit exists, 
where $B(x,r)$ denotes the open ball of center $x$ and radius $r$.

The geodesic flow on the unit tangent bundle of a strictly convex real projective surface is $C^1$ and Anosov. Therefore a classical theorem of Bowen \cite{Bowen} applies to conclude that
the topological entropy $h_{\mathrm{vol}}$ is equal to the growth rate of lengths of closed geodesics. More precisely, if $N(T)$ denotes the number of closed geodesics in $\Omega/\Gamma$ and $h_{\mathrm{top}}$ the topological entropy of the flow, we have 
\[N(T)\simeq \frac{e^{h_{\mathrm{top}}T}}{h_{\mathrm{top}}T}.
\]

It is shown in \cite{crampon} that
$$0< h_{\mathrm{vol}}\leq h_{\mathrm{hyp}}=1$$ where $h_{\mathrm{hyp}}$ is the topological entropy of a hyperbolic structure. 

Now we summarize some general facts about the theory of H\"older cocycles and Patterson-Sullivan measures associated to real projective structures.

Let $\Gamma$ be a properly discontinuous action of a group of projective transformations on a convex set $\Omega$.
 A family of finite Borel measures $\{\nu_x\}_{x\in \Omega}$ defined  on $\partial{ \Omega}$ is said to be 
  an \emph{$\alpha$-conformal density}\index{conformal density}\index{$\alpha$-conformal density} (or a \emph{conformal density of dimension $\alpha$}) for $\Gamma$  if any two metrics in this family are equivalent (that is, if they have the same measure-zero sets) and if they satisfy the
 following properties:
 \begin{enumerate}
 \item \label{P1} { $\displaystyle \frac{d\nu_y}{d\nu_x}(\zeta)=e^{-\alpha B_{\zeta}(x,y)}$ where
 $B_{\zeta}(x,\cdot)$ is the Busemann function based at $\zeta$ such that
 $B_{\zeta}(x,x)=0 \ \forall x,y \in \Omega$ ;}
\item{ $\displaystyle \frac{d\nu_{\gamma x}}{d\nu_x}(\zeta)=\frac{d\nu_x}{d\nu_{\gamma^{-1} x}}
 (\gamma^{-1} \zeta)  \ \forall x,y \in \Omega$ and $\gamma\in \Gamma$;}
 \item $\gamma_*\nu_x=\nu_{\gamma x}  \ \forall x \in \Omega$ and $\gamma\in \Gamma$.
\end{enumerate}
Condition (\ref{P1}) expresses the fact that the Radon-Nikodym cocycles $\displaystyle \frac{d\nu_y}{d\nu_x}(\zeta)$ of the family of measures $\{\nu_x\}_{x\in \Omega}$ are (up to a constant) equal to the Busemann cocycles of the convex set $\Omega$. We refer the reader to the paper \cite{Pat} of Patterson and \cite{Sulli} of Sullivan for the original ideas behind the introduction of conformal densities.

Given an $\alpha$-conformal density $\{\nu_x\}_{x\in \Omega}$ on $\partial{ \Omega}$, the measure
$$dU(\zeta,\eta)=dU_x(\zeta,\eta)=e^{2\delta(\Gamma)(\zeta,\eta)_x}d\nu_x(\zeta)d\nu_x(\eta)$$
is a $\Gamma$-invariant measure on $\partial \Omega\times\partial \Omega$ which is
independent of $x\in \Omega$, where $(\zeta,\eta)_x$ denotes the quantity
\[(\zeta,\eta)_x=-B_\zeta(x,z)-B_\eta(x,z)
\]
 for any $z$ in $\Omega$ and where $\delta(\Gamma)$ is equal to the volume entropy of the associated Hilbert metric. (The quantity $(\zeta,\eta)_x$ is also related to the Gromov product of $\zeta$ and $\eta$ with basepoint $x$.)

 Finally
$$dUdt$$ is a geodesic flow invariant measure on $T_1M$, where
$M=\Omega/\Gamma$, and it is called a \emph{Bowen-Margulis measure}.\index{measure!Bowen-Margulis}\index{Bowen-Margulis measure} In the classical theory of dynamical systems, the Bowen-Margulis measure is an invariant 
measure for a hyperbolic system (in particular, for an Anosov flow). In the case of a mixing Anosov flow (like the geodesic flow of a compact hyperbolic manifold), the Bowen-Margulis measure maximizes entropy. In this sense, it provides a very good description of the complexity of the dynamical system. There is a construction of the Bowen-Margulis measure using the Patterson-Sullivan techniques and we shall talk about this below.

Since we are interested in the group $\Gamma$ itself rather than $\Omega$, a Patterson-Sullivan measure is
$$\frac{d\gamma_*\mu_o}{d\mu_o}(x)=e^{-h_c c(\gamma^{-1},x)}$$ for some H\"older cocycle $c$. Henceforth we will omit the base point o.
Two measures $\mu$ and $\mu'$ are equivalent if and only if two associated cocycles  $c$ and $c'$ have the same periods \cite{Le}.
The Patterson-Sullivan geodesic current\index{Patterson-Sullivan geodesic current} associated to the H\"older cocycle $c$ is
$$dm_c(x,y)=e^{2h_c(x,y)_o} d\mu(x)d\mu(y).$$ In \cite{Le}, it is shown that there is a 1-1 correspondence between H\"older cocycles and Patterson-Sullivan geodesic currents.

From now on, we denote by $\cal G$ the set of Patterson-Sullivan geodesic currents defined by H\"older cocycles. The space of geodesic currents, as a space of measures, is equipped with a natural weak${}^*$ topology of convergence on continuous functions. 
Note that if $\rho^*$ is a dual structure of $\rho$, since $\rho^*$ has the same marked length spectrum as $\rho$, the cocycle defined by $\rho^*$ is the same as the one defined by $\rho$.  By associating a H\"older cocycle $c_\rho=B_{(\cdot)}(o,\cdot)$ and the Patterson-Sullivan geodesic current $\mu_\rho$ to the projective structure $\rho$, we obtain a map from the moduli space $\cal D$ of strictly convex structures on $S$
$$P:\cal D\ra \cal G.$$ This map is 1-1 on Teichm\"uller space and 2-1 elsewhere \cite{kim1}. The preimage of a point consists in two dual projective structures.
The interested reader will find more details in  \cite{kim4}.

\section{Compactification of the deformation space of convex real projective structures}

Extensive work has been done on the compactification of Teichm\"uller space and of Riemann's moduli space, and more recently on that of the character varieties of representations of fundamental groups of surfaces into various Lie groups.  Several possible compactifications and boundary constructions of Teichm\"uller space have been obtained. Some of them use hyperbolic geometry (Thurston's compactification, the horofunction boundary construction, etc.), others compactifications use complex structures (Teichm\"uller's compactification, the Bers compactification, etc.), and others use algebraic geometry (e.g. the Morgan-Shalen compactification). There are even others. Each compactification captures some essential properties of the space. It is then natural to try to construct compactifications of moduli spaves projective structures. 

  A compactification of the deformation space of strictly convex real projective structures on a manifold has been constructed by the first author of this paper in \cite{kim2}. It is related to  Hilbert geometry in the sense that it uses the length spectrum of this metric, in the same way as Thurston's geometric compactification $\overline {\cal T}$ of Teichm\"uller space $\cal T$ uses the hyperbolic length spectrum. The compactification of the deformation space $\cal D$ of strictly convex real projective structures is however a bit mysterious compared to Thurston's compactification of Teichm\"uller space. 
  Loftin developed in \cite{Lo1} another compactification of $\cal D$ which is based on holomorphic cubic differentials on degenerate (noded) surfaces.

  Since Teichm\"uller space $\cal T$ is a subspace of $\cal D$, one can naturally expect that the compactification of $\cal D$ obtained by using the length spectrum should include the Thurston boundary consisting of projective measured laminations or, equivalently, of actions of the fundamental group of the surface on $\mathbb{R}$-trees. The compactification in \cite{kim2} is done in terms of the geometry of $X=\mathrm{SL}(3,\R)/\mathrm{SO}(3)$ by regarding a real projective structure as a
holonomy representation $\rho:\pi_1(S)\ra \mathrm{SL}(3,\R)$. A boundary point of the compactification
is then either a reducible representation or a limit representation which acts on the asymptotic cone of $X$.
Specifically if $\rho_i:\pi_1(S)\ra \mathrm{SO}(2,1)\subset \mathrm{SL}(3,\R)$ is a sequence of hyperbolic structures which
converges to a projective lamination $\lambda$ in Thurston's compactification, the limit action of $\rho_i$ converges to the affine building $\R\times T_\lambda$, where $T_\lambda$ is the real tree dual to the measured lamination $\lambda$.
In this way, the space of projective measured laminations appears naturally as a subset of the boundary of the deformation space of strictly convex real projective structures.

On the other hand, the geodesic currents introduced in Section \ref{geodesiccurrents} are natural generalizations of measured laminations.
Hence we may also use the space $\mathcal{G}$ of geodesic currents to compactify $\cal D$.  On $\cal G$, one can define an intersection form
$$i:\cal G\times \cal G\ra \R^+,$$ so that $i(\alpha,\beta)$  is the total mass of the product measure $\alpha\times\beta$ on the set of pairs of transversal geodesics on $S$. This is a natural generalization of the intersection form $i(\gamma_1,\gamma_2)$ when $\gamma_1$ and $\gamma_2$ are measured geodesic laminations.  The geodesic current $\mu_m$ associated to the Hilbert metric $m$ of a strictly convex real projective structure is properly normalized as:
$$i(\mu_m,\mu_m)=\frac{\pi}{2}\mathrm{Area}(m),$$ like in the hyperbolic case \cite{Bo}.
This compactification by geodesic currents satisfies the  following (see \cite{kim4}):
\begin{thm}The two compactifications of the space of strictly convex
real projective structures on a closed surface $S$, the one defined via the marked
length spectrum and the other via geodesic currents,  are naturally
homeomorphic.
\end{thm}
\begin{proof}[Sketch of proof] Look at the following diagram
$$
\xymatrix{
\cal D \ar[r]^{L}\ar[rd]_{\mathbb P P} & {\mathbb P( \mathbb {R}_+}^{\cal S})  \\
& \mathbb P\cal G \ar[u]^{\ell}}
\;\;
$$
where $\cal S$ is the set of conjugacy classes of elements in $\pi_1(S)$ and $L$ is the marked length  spectrum map relatively to the Hilbert metric. The map $\ell$ is defined using the periods of the corresponding H\"older cocycles.  The diagram
 commutes, and the map $\ell$ is injective, hence the compactifications of the images of $L$ and $\mathbb P P$ are homeomorphic.
\end{proof}

\begin{co}  Thurston's compactification $\overline{\cal T}$ by projective measured laminations is contained in
$\overline{\mathbb P P(\cal D)}$.
\end{co}
\begin{proof}
When restricted to $\cal T$, if $\lambda_i\mu_{m_i}\ra \alpha$ in $\cal G$ and the projective structures $m_i$ diverge,
then by the properness of the map $\cal T\ra \cal G$, $\lambda_i\ra 0$ and
$$i(\alpha,\alpha)=\lim_{i\ra\infty} i(\lambda_i\mu_{m_i},\lambda_i\mu_{m_i})=\lim_{i\ra\infty}\lambda_i^2 i(\mu_{m_i},\mu_{m_i})=\lim_{i\ra\infty}\lambda_i^2 \frac{\pi}{2}2\pi|\chi(S)|=0.$$ Hence $\alpha$ must be a measured lamination since
 measured laminations are characterized by self-intersection number being zero.
\end{proof}
In this way, we see again Thurston's compactification sitting inside the compactification of real projective structures.
We believe that whenever the Hilbert metric area of a diverging sequence is bounded above, the sequence will converge to a projective measured lamination.

%
%

\end{document}